\documentclass[11pt]{amsart}

\usepackage[T1]{fontenc}
\usepackage[utf8]{inputenc}
\usepackage{lmodern}
\usepackage{microtype}
\usepackage{amsmath,amssymb,amsfonts,amsthm,mathtools}
\usepackage{enumitem}
\usepackage[colorlinks=true,citecolor=blue,linkcolor=blue,urlcolor=blue]{hyperref}

\newtheorem{theorem}{Theorem}[section]
\newtheorem{proposition}[theorem]{Proposition}
\newtheorem{lemma}[theorem]{Lemma}

\newtheorem*{theoremA}{Theorem A}
\newtheorem*{theoremB}{Theorem B}
\newtheorem*{corollaryC}{Corollary C}
\theoremstyle{definition}
\newtheorem{definition}[theorem]{Definition}

\newcommand{\bN}{\mathbb{N}}

\newcommand{\cD}{\mathcal{D}}
\newcommand{\cE}{\mathcal{E}}

\newcommand{\cQ}{\mathcal{Q}}

\newcommand{\C}{\mathbb C}
\newcommand{\N}{\mathbb N}
\newcommand{\R}{\mathbb R}
\newcommand{\Qrat}{\mathbb Q}
\newcommand{\Z}{\mathcal Z}
\newcommand{\AT}{\operatorname{AT}}
\newcommand{\UAT}{\operatorname{UAT}}
\newcommand{\QD}{\operatorname{QD}}
\newcommand{\LFD}{\operatorname{LFD}}
\newcommand{\Tr}{\operatorname{Tr}}
\newcommand{\tr}{\operatorname{tr}}
\newcommand{\dist}{\operatorname{dist}}
\newcommand{\diag}{\operatorname{diag}}
\newcommand{\soc}{\operatorname{soc}}
\newcommand{\id}{\operatorname{id}}
\newcommand{\KK}{\operatorname{KK}}
\newcommand{\otmax}{\otimes_{\max}}
\newcommand{\otmin}{\otimes_{\min}}

\title[Locally finite-dimensional traces II]{On locally finite-dimensional traces II}
\author{Mehdi Moradi}
\address{Department of Mathematics and Statistics, University of Ottawa, K1N 6N5, Canada}
\email{mmoradi@uottawa.ca}
\author{Massoud Amini}
\address{Faculty of Mathematical Sciences, Tarbiat Modares University, Tehran 144115-134, Iran}
\email{mamini@modares.ac.ir}


\begin{document}
\begin{abstract}
We continue the study of locally finite-dimensional traces introduced in our earlier work. We give new characterizations of LFD traces in terms of finite-rank projections in irreducible representations and in the socle of the bidual. We show that, for separable nowhere scattered \(C^*\)-algebras, the set of all LFD traces is convex. We also prove that quasidiagonal traces form a face of the trace simplex and record applications to strongly self-absorbing \(C^*\)-algebras. We construct an exact tracially AF algebra not KK-equivalent to any nuclear \(C^*\)-algebra.
\end{abstract}

\maketitle

\section{Introduction}

The classification programme for C*-algebras depends heavily on understanding when infinite-dimensional objects can be recovered from finite-dimensional data. Traces are central in this problem: they encode the finite, stably finite part of the theory, and the way a trace is approximated by matrices often reflects deep structure of the ambient algebra. The study of finite-dimensional approximation phenomena goes back to Halmos's work on quasidiagonal operators \cite{Halmos1970}, Thayer's work on quasidiagonal C*-algebras \cite{Thayer1977}, and Voiculescu's intrinsic approach to quasidiagonality \cite{Voiculescu1991}. Popa's local finite-dimensional approximation property \cite{Popa1997}, Lin's tracial approximation theory \cite{LinTAF2001}, and Brown's tracial approximation classes \cite{BrownMemoir} provide the framework used here. The theorem of Tikuisis--White--Winter that faithful traces on separable nuclear UCT algebras are quasidiagonal is one of the main modern results connecting traces, quasidiagonality, and classification \cite{TWW2017}.

Amenable, quasidiagonal, and locally finite-dimensional traces form a natural hierarchy of tracial approximation properties. The locally finite-dimensional condition, introduced by Brown and studied further in \cite{AminiMoradi2023,BrownMemoir}, asks not only that the trace be approximated by finite-dimensional maps, but also that the algebra itself be approximated by the multiplicative domains of those maps. This extra multiplicative-domain requirement is the main source of difficulty. It is also what makes LFD traces particularly useful in connection with tracially AF algebras and classification.

The first goal of this paper is to replace the multiplicative-domain formulation by a more geometric one. We characterize LFD traces using finite-rank projections in irreducible representations, and equivalently using projections in the socle of the bidual. This viewpoint makes the distance to the multiplicative domain visible through commutators with finite-rank projections. It also clarifies the role of disjointness: when several finite-dimensional summands are used simultaneously, the relevant projections must be centrally disjoint in the bidual, not merely orthogonal as projections.

We then use this characterization to prove a convexity result. Thiel and Vilalta's theory of nowhere scattered C*-algebras \cite{ThielVilalta2024} gives exactly the representation-theoretic flexibility needed to duplicate finite-dimensional compression data in disjoint irreducible representations. Consequently, for separable nowhere scattered algebras, the LFD traces form a convex subset of the trace simplex.

The second theme is quasidiagonality of traces. Winter asked whether quasidiagonal traces form a face of the tracial state space \cite{Winter2016}. We prove that they do. The proof is based on cutting a quasidiagonal trace by central projections in its GNS von Neumann algebra. The nontrivial point is that the resulting central projection in a tracial ultraproduct must be lifted to a genuine projection in a norm-ultraproduct relative commutant. This uses Kirchberg's central surjectivity theorem, real-rank-zero of the relevant relative commutant, and projection lifting for quotients of real-rank-zero C*-algebras.

The third theme is uniformity. We prove that, for an irreducibly represented quasidiagonal C*-algebra containing no nonzero compact operators, every quasidiagonal trace is LFD. If the algebra is locally reflexive, then LFD traces are in fact uniformly LFD. The proof combines Brown's finite-rank projection theorem for quasidiagonal traces with local reflexivity and a subtracial normality argument, upgrading convergence from the algebra to the bidual.

These tracial results have two classification-theoretic consequences. First, they give a rigidity statement for strongly self-absorbing algebras: if \(\cD\) is strongly self-absorbing and \(K_*(\cD)\cong K_*(\cQ)\), where \(\cQ\) is the universal UHF algebra, then \(\cD\cong \cQ\) exactly when \(\cD\) is quasidiagonal. Second, they reformulate the UCT problem in terms of strongly self-absorbing uniqueness among simple nuclear tracially AF algebras with the K-theory of \(\cQ\).

We finish with a contrasting example. Starting from the Skandalis obstruction, in the form recalled by Higson--Guentner and sharpened rationally by Puschnigg \cite{Skandalis1988,HigsonGuentner2004,Puschnigg2018}, and then using Dadarlat's construction as developed by Niu--Wang \cite{DadarlatAF2000,NiuWang2021}, we obtain an exact separable unital tracially AF algebra which is not KK-equivalent to any nuclear C*-algebra and is not Jiang--Su stable. This example is placed at the end of the paper because it uses the approximation theory developed earlier but points in the opposite direction: tracial AF approximation alone does not force nuclear-type KK-behaviour.

The paper is organized as follows. Section 2 contains the bidual characterization of LFD traces, the convexity theorem for nowhere scattered algebras, and the faciality theorem for QD traces. Section 3 proves the uniform local finite-dimensionality result. Section 4 gives the applications to strongly self-absorbing algebras and the UCT reformulation. Section 5 constructs the Skandalis-type example.

\section{Locally finite-dimensional traces revisited}

\subsection{Preliminaries}

We shall use the following convention throughout. Projections \(p_1,\ldots,p_r\in\soc(A^{**})\) are called \emph{centrally disjoint} if their central supports are pairwise orthogonal in \(Z(A^{**})\). Equivalently, the finite-dimensional corners \(p_jA^{**}p_j\) live in mutually disjoint central summands of \(A^{**}\).

In this case, if \(P=p_1+\cdots+p_r\), then for every \(a\in A\),
\[
 [a,P]=0\quad\Longleftrightarrow\quad [a,p_j]=0\quad (1\le j\le r).
\]
Indeed, if \(z_j=z(p_j)\), then \(p_j=z_jP\), and since \(z_j\) is central,
\[
 [a,p_j]=[a,z_jP]=z_j[a,P].
\]
The converse follows by summing over \(j\). This central-disjointness hypothesis is essential: mere pairwise orthogonality of projections does not imply the displayed equivalence.

We begin by recalling Brown's tracial approximation classes and by fixing the representation-theoretic language used throughout the paper. The definitions are standard, but we spell them out because the distinction between quasidiagonal and locally finite-dimensional approximation is central to the arguments below.

In this paper traces are tracial states, and \(T(A)\) denotes the tracial state space of a C*-algebra \(A\). We write \(M_n\) for the algebra of \(n\times n\) complex matrices, \(\Tr\) for its non-normalized trace, and \(\tr_n\) for its normalized trace. For \(x\in M_n\), put \(\|x\|_2=\tr_n(x^*x)^{1/2}\). If \(\rho:A\to B(H_\rho)\) is a representation, we denote its central support by \(c(\rho)\). A projection \(p\in A^{**}\) belongs to \(\soc(A^{**})\) if \(pA^{**}p\) is finite-dimensional.

\begin{definition}\label{def:traces}
Let \(A\) be a C*-algebra and \(\tau\in T(A)\).
\begin{enumerate}[label=\textup{(\roman*)}]
\item \(\tau\) is \emph{amenable} if there is a net of c.c.p. maps \(\phi_i:A\to M_{k(i)}\) such that
\[
 \|\phi_i(ab)-\phi_i(a)\phi_i(b)\|_2\to 0,
 \qquad \tau(a)=\lim_i\tr_{k(i)}(\phi_i(a))
\]
for all \(a,b\in A\).
\item \(\tau\) is \emph{quasidiagonal}, or QD, if the same conditions hold with operator norm multiplicativity:
\[
 \|\phi_i(ab)-\phi_i(a)\phi_i(b)\|\to 0.
\]
\item \(\tau\) is \emph{locally finite-dimensional}, or LFD, if there is a net of c.c.p. maps \(\phi_i:A\to M_{k(i)}\) such that
\[
 \dist(a,A_{\phi_i})\to 0,
 \qquad \tau(a)=\lim_i\tr_{k(i)}(\phi_i(a))
\]
for all \(a\in A\), where \(A_{\phi_i}\) denotes the multiplicative domain of \(\phi_i\).
\end{enumerate}
When the convergence of the trace functionals is in norm in \(A^*\), we add the adjective \emph{uniform}. We write \(\AT(A)\), \(\AT(A)_{\QD}\), and \(\AT(A)_{\LFD}\) for the amenable, QD, and LFD traces, and \(\UAT(A)\), \(\UAT(A)_{\QD}\), and \(\UAT(A)_{\LFD}\) for the corresponding uniform classes.
\end{definition}

\begin{definition}\label{def:nowhere-scattered}
Following Thiel--Vilalta \cite{ThielVilalta2024}, a C*-algebra \(A\) is called
\emph{nowhere scattered} if none of its quotients contains a minimal open projection.
Here an open projection \(p\in A^{**}\) is minimal if its associated hereditary
subalgebra is one-dimensional, equivalently
\[
 pA^{**}p\cap A=\C p .
\]
An \emph{ideal-quotient}, or subquotient, of \(A\) means a quotient \(I/J\), where
\(J\triangleleft I\triangleleft A\) are closed two-sided ideals. The main
characterization of Thiel--Vilalta says that, for a C*-algebra \(A\), the following
conditions are equivalent:
\begin{enumerate}[label=\textup{(\roman*)}]
\item \(A\) is nowhere scattered;
\item \(A\) has no nonzero elementary ideal-quotient;
\item \(A\) has no nonzero scattered ideal-quotient;
\item \(A\) has no nonzero type I ideal-quotient.
\end{enumerate}
We use this characterization only to ensure that finite-dimensional compression
data can be duplicated in disjoint irreducible representations, without producing
an isolated finite-dimensional subquotient.
\end{definition}

The next proposition is the basic technical tool of the paper. It translates the definition of an LFD trace into finite-rank compression data. The point is that the multiplicative domain of a compression is controlled by commutation with the underlying projection. When several summands are present, central disjointness allows us to reduce commutation with all summands to a single Blackadar--Kirchberg distance estimate.

\begin{proposition}\label{prop:char-LFD}
Let \(A\) be a C*-algebra and let \(\tau\in T(A)\). The following are equivalent.
\begin{enumerate}[label=\textup{(\alph*)}]
\item \(\tau\in\AT(A)_{\LFD}\).
\item For every \(\varepsilon>0\) and every finite set of contractions \(F\subset A\), there are positive integers \(k_1,\ldots,k_r\), pairwise disjoint irreducible representations
\[
 \pi_j:A\to B(H_j),\qquad 1\le j\le r,
\]
and finite-rank projections \(p_j\in B(H_j)\), such that
\[
 \|[p_j,\pi_j(a)]\|<\varepsilon
 \qquad (a\in F,
 1\le j\le r),
\]
and
\[
 \left|
 \tau(a)-
 \frac{\sum_{j=1}^rk_j\Tr(p_j\pi_j(a))}{\sum_{j=1}^rk_j\Tr(p_j)}
 \right|<\varepsilon
 \qquad (a\in F).
\]
\item For every \(\varepsilon>0\) and every finite set of contractions \(F\subset A\), there are positive integers \(k_1,\ldots,k_r\) and centrally disjoint finite-rank projections
\[
 p_1,\ldots,p_r\in\soc(A^{**})
\]
such that
\[
 \|[p_j,a]\|<\varepsilon
 \qquad (a\in F,
 1\le j\le r),
\]
and
\[
 \left|
 \tau(a)-
 \frac{\sum_{j=1}^rk_j\Tr(p_ja)}{\sum_{j=1}^rk_j\Tr(p_j)}
 \right|<\varepsilon
 \qquad (a\in F).
\]
\end{enumerate}
\end{proposition}

\begin{proof}
\((a)\Rightarrow(b)\). Let \(F\subset A\) be finite and let \(\varepsilon>0\). Choose an LFD c.c.p. map \(\phi:A\to M_N\) and, for each \(a\in F\), choose \(b_a\in A_\phi\) with \(\|a-b_a\|\) sufficiently small. The restriction
\[
 \phi|_{A_\phi}:A_\phi\to M_N
\]
is a finite-dimensional representation. Decompose it as
\[
 \phi|_{A_\phi}\cong
 (\sigma_1^{\oplus k_1})\oplus\cdots\oplus(\sigma_r^{\oplus k_r}),
\]
where the \(\sigma_j\)'s are irreducible and pairwise inequivalent after grouping. By the extension theorem for irreducible representations of C*-subalgebras \cite[Theorem 5.5.1]{Murphy1990}, each \(\sigma_j\) extends to an irreducible representation \(\pi_j:A\to B(H_j)\). After grouping extensions with the same central support, we may assume the \(\pi_j\)'s are pairwise disjoint. Let \(p_j\) be the projection onto the finite-dimensional Hilbert space on which \(\sigma_j\) acts. Since \(b_a\in A_\phi\), we have \([p_j,\pi_j(b_a)]=0\), hence
\[
 \|[p_j,\pi_j(a)]\|
 \le 2\|a-b_a\|.
\]
Taking \(\|a-b_a\|\) sufficiently small gives the commutator estimate. The trace estimate follows from the decomposition of \(\phi|_{A_\phi}\) and the approximation of \(a\) by \(b_a\).

\((b)\Rightarrow(c)\). Identify each irreducible summand \(B(H_j)=\pi_j(A)''\) with the central summand \(c(\pi_j)A^{**}\). Since the \(\pi_j\)'s are pairwise disjoint, their central supports are orthogonal. The finite-rank projections \(p_j\) then determine centrally disjoint projections in \(\soc(A^{**})\), and the estimates are unchanged.

\((c)\Rightarrow(a)\). Define
\[
 \phi:A\to
 M_{k_1}(p_1A^{**}p_1)\oplus\cdots\oplus M_{k_r}(p_rA^{**}p_r)
\]
by
\[
 \phi(a)=
 \diag(p_1ap_1,\ldots,p_1ap_1)\oplus\cdots\oplus
 \diag(p_rap_r,\ldots,p_rap_r),
\]
where the \(j\)-th block is repeated \(k_j\) times. For a single compression \(a\mapsto pap\),
\[
 pa^*ap-(pap)^*(pap)=pa^*(1-p)ap,
\]
and similarly
\[
 paa^*p-(pap)(pap)^*=pa(1-p)a^*p.
\]
Thus \(a\) lies in the multiplicative domain of \(a\mapsto pap\) if and only if \(ap=pa\). Hence
\[
 A_\phi=\{a\in A:[a,p_j]=0,\ 1\le j\le r\}.
\]
Since the \(p_j\)'s are centrally disjoint, this is the same as
\[
 A_\phi=\{a\in A:[a,P]=0\},
 \qquad P=p_1+\cdots+p_r.
\]
By the Blackadar--Kirchberg distance formula \cite[Corollary 3.5]{BlackadarKirchberg2001}, also recorded in \cite[Proposition 11.3.6]{BrownOzawa2008},
\[
 \dist(a,A_\phi)=\dist(a,A_P)=\|[a,P]\|
 =\max_{1\le j\le r}\|[a,p_j]\|.
\]
Together with the trace estimate in \((c)\), this gives an LFD approximation to \(\tau\).
\end{proof}

\begin{lemma}
Let \(A\) be separable and nowhere scattered. Let \(\pi:A\to B(H)\) be an irreducible representation, let \(p\in B(H)\) be a finite-rank projection, let \(F\subset A\) be finite, and let \(\delta>0\). For every \(N\in\N\), there are pairwise disjoint irreducible representations
\[
 \pi^{(1)},\ldots,\pi^{(N)}
\]
and finite-rank projections \(p^{(m)}\in B(H^{(m)})\) with
\[
 \operatorname{rank}(p^{(m)})=\operatorname{rank}(p),
\]
such that, for every \(a\in F\) and every \(m\),
\[
 \left|
 \Tr(p^{(m)}\pi^{(m)}(a))-
 \Tr(p\pi(a))
 \right|<\delta\operatorname{rank}(p),
\]
and
\[
 \|[p^{(m)},\pi^{(m)}(a)]\|
 <\|[p,\pi(a)]\|+\delta.
\]
\end{lemma}

\begin{proof}
The compression \(a\mapsto p\pi(a)p\) is a pure matricial state. A standard consequence of nowhere scatteredness is that pure matricial states are weak-* limit points of pure matricial states coming from pairwise disjoint irreducible representations; equivalently, after excluding any prescribed finite family of primitive ideals, one can approximate the compression on a prescribed finite set. This is the non-simple version of the disjointness argument used in \cite{AminiMoradi2023}, with nowhere scatteredness supplied by the Thiel--Vilalta characterization \cite{ThielVilalta2024}.

Applying this fact inductively, choose \(\pi^{(1)},\ldots,\pi^{(N)}\), avoiding the previously chosen central supports at each step, and choose finite-dimensional subspaces carrying compressions within a small tolerance \(\eta\) of \(p\pi(\cdot)p\) on \(F\). Let \(p^{(m)}\) be the projection onto the chosen subspace. The trace estimate follows from
\[
 \left|
 \Tr(p^{(m)}\pi^{(m)}(a))-\Tr(p\pi(a))
 \right|
 \le \operatorname{rank}(p)\,
 \|p^{(m)}\pi^{(m)}(a)p^{(m)}-p\pi(a)p\|,
\]
and the commutator estimate follows by comparing the two compressed matrix systems and then taking \(\eta\) sufficiently small relative to \(\delta\).
\end{proof}

We now turn to convexity. In the simple case this was proved in \cite{AminiMoradi2023}. The argument only needs enough irreducible representations to separate repeated copies of a given finite-dimensional compression. Nowhere scatteredness supplies precisely this substitute for simplicity.

\begin{theorem}\label{thm:convex-LFD}
Let \(A\) be a separable nowhere scattered C*-algebra. Then \(\AT(A)_{\LFD}\) is convex.
\end{theorem}

\begin{proof}
We follow the convexity argument of \cite{AminiMoradi2023}, replacing simplicity by nowhere scatteredness. It is enough to prove midpoint convexity and then use weak-* closedness of the LFD condition. Let \(\tau_0,\tau_1\in\AT(A)_{\LFD}\), let \(F\subset A\) be a finite set of contractions, and let \(\varepsilon>0\). Choose \(\eta>0\) small enough so that all accumulated errors below are less than \(\varepsilon\).

By Proposition \ref{prop:char-LFD}, choose finite-dimensional compression functionals \(f_0,f_1\) with
\[
 |f_i(a)-\tau_i(a)|<\eta
 \qquad (a\in F,
 i=0,1),
\]
where
\[
 f_0(a)=
 \frac{\sum_{j=1}^r\Tr(p_j\pi_j(a))}{D},
 \qquad
 D=\sum_{j=1}^r\Tr(p_j),
\]
and
\[
 f_1(a)=
 \frac{\sum_{\ell=1}^s\Tr(q_\ell\rho_\ell(a))}{E},
 \qquad
 E=\sum_{\ell=1}^s\Tr(q_\ell).
\]
Multiplicity coefficients are absorbed by repeating summands.

Use the disjoint duplication lemma to replace the family representing \(f_0\) by \(E\) disjoint copies and the family representing \(f_1\) by \(D\) disjoint copies, all mutually disjoint, with the compressed matrix coefficients changed by less than \(\eta\) on \(F\). Denote these copies by
\[
 (p_{j,m},\pi_{j,m}),\qquad 1\le j\le r,
 1\le m\le E,
\]
and
\[
 (q_{\ell,n},\rho_{\ell,n}),\qquad 1\le \ell\le s,
 1\le n\le D.
\]
Set
\[
 h(a)=\frac{
 \sum_{j=1}^r\sum_{m=1}^E\Tr(p_{j,m}\pi_{j,m}(a))+
 \sum_{\ell=1}^s\sum_{n=1}^D\Tr(q_{\ell,n}\rho_{\ell,n}(a))}{2DE}.
\]
Then \(h\) is represented by pairwise disjoint finite-rank compressions and has the same commutator control as the original two systems, up to the error \(\eta\). Moreover,
\[
 \left|h(a)-\frac{f_0(a)+f_1(a)}2\right|<\eta
 \qquad (a\in F).
\]
Therefore
\[
 \left|h(a)-\frac{\tau_0(a)+\tau_1(a)}2\right|<3\eta
 \qquad (a\in F).
\]
Proposition \ref{prop:char-LFD} gives
\[
 \frac{\tau_0+\tau_1}{2}\in\AT(A)_{\LFD}.
\]
Iterating the midpoint argument gives closure under dyadic convex combinations. Finally, \(\AT(A)_{\LFD}\) is weak-* closed: for fixed \(F\) and \(\varepsilon\), a weak-* limit of LFD traces is approximated on \(F\) by first choosing a nearby LFD trace and then using its finite-dimensional approximation. Hence \(\AT(A)_{\LFD}\) is convex.
\end{proof}

\subsection{The face of quasidiagonal traces}

We next prove that QD traces form a face. The idea is simple at the level of traces: if a QD trace is decomposed as a convex combination, the summands are obtained by cutting the GNS von Neumann algebra by central positive contractions. The difficulty is to show that these cut traces still admit norm-multiplicative matrix approximations. This is done first for central projections, and then for central positive contractions by spectral approximation.

We use a central surjectivity input. Let \(\omega\) be a free ultrafilter and let \((B_n,\tau_n)\) be separable unital tracial C*-algebras. Put
\[
 B=\prod_\omega B_n,
 \qquad
 M=\prod_\omega \pi_{\tau_n}(B_n)''.
\]
Let \(q_\omega:B\to M\) be the quotient map.

\begin{lemma}\label{lem:central-surj}
For every separable C*-subalgebra \(S\subset B\), the sequence
\[
 0\longrightarrow \ker(q_\omega)\cap S'
 \longrightarrow B\cap S'
 \longrightarrow M\cap q_\omega(S)'
 \longrightarrow 0
\]
is exact.
\end{lemma}

\begin{proof}
This is Kirchberg's central surjectivity theorem in the form used in \cite{KirchbergRordam2014}. The kernel of \(q_\omega\) is a \(\sigma\)-ideal, and the conclusion follows from the standard lifting theorem for relative commutants modulo \(\sigma\)-ideals.
\end{proof}

\begin{lemma}\label{lem:central-rn}
Let \(A\) be a C*-algebra, let \(\tau\in T(A)\), and put \(N=\pi_\tau(A)''\). If \(\rho\) is a positive tracial functional on \(A\) with \(0\le \rho\le \tau\), then there is a unique central positive contraction \(h_\rho\in Z(N)\) such that
\[
 \rho(a)=\tau^{**}(h_\rho\pi_\tau(a))
 \qquad (a\in A).
\]
In particular, if \(\tau=s\gamma_1+t\gamma_2\), where \(s,t>0\) and \(s+t=1\), then there are central positive contractions \(h_1,h_2\in Z(N)\), with \(h_1+h_2=1\), such that
\[
 s\gamma_1(a)=\tau^{**}(h_1\pi_\tau(a)),
 \qquad
 t\gamma_2(a)=\tau^{**}(h_2\pi_\tau(a)).
\]
\end{lemma}

\begin{proof}
This is the central Radon--Nikodym theorem for traces in the GNS von Neumann algebra; see \cite[Lemma 6.5.4]{Dixmier1983}. Since \(\rho\le\tau\), \(\rho\) extends to a normal positive functional on \(N\) dominated by \(\tau^{**}\). The usual noncommutative Radon--Nikodym theorem gives a positive contraction \(h_\rho\in N\). Traciality of \(\rho\) forces \(h_\rho\) to commute with \(\pi_\tau(A)\), hence with \(N\), so \(h_\rho\in Z(N)\).
\end{proof}

\begin{lemma}\label{lem:subtracial-normality}
Let \(M\) be a von Neumann algebra, let \(\rho\) be a normal positive functional on \(M\), and let \(\psi:M\to M_k\) be completely positive. Suppose that \(\psi\) is subtracial with respect to \(\rho\), meaning that for some \(C>0\),
\[
 \Tr_k(\psi(x))\le C\rho(x)
 \qquad (x\in M_+).
\]
Then \(\psi\) is normal.
\end{lemma}

\begin{proof}
For every \(\xi\in\C^k\), the functional
\[
 x\mapsto \langle \psi(x)\xi,\xi\rangle
\]
is positive and satisfies
\[
 \langle \psi(x)\xi,\xi\rangle
 \le \|\xi\|^2\Tr_k(\psi(x))
 \le C\|\xi\|^2\rho(x)
 \qquad (x\in M_+).
\]
A positive functional dominated by a normal positive functional is normal. Hence all matrix coefficients of \(\psi\) are normal, and therefore \(\psi\) is normal.
\end{proof}

The following lemma is the heart of the faciality argument. In the proof a central projection in the tracial ultraproduct relative commutant is lifted to a projection in the norm ultraproduct relative commutant, and the resulting matrix corners give the desired QD approximation.

\begin{lemma}\label{lem:central-cut-QD}
Let \(A\) be separable, let \(\tau\in\AT(A)_{\QD}\), and put
\(N=\pi_\tau(A)''\). If \(P\in Z(N)\) is a central projection with
\(\alpha=\tau^{**}(P)>0\), then the normalized trace
\[
 \tau_P(a)=\frac{\tau^{**}(P\pi_\tau(a))}{\alpha},\qquad a\in A,
\]
is quasidiagonal.
\end{lemma}

\begin{proof}
Choose u.c.p. maps \(\phi_n:A\to M_{k(n)}\) witnessing quasidiagonality of
\(\tau\), and let
\[
 B_\omega=\prod_\omega M_{k(n)},\qquad
 M_\omega=\prod_\omega^{\tr}M_{k(n)}
\]
be the norm and tracial ultraproducts. Let
\(q_\omega:B_\omega\to M_\omega\) be the quotient map. The maps \(\phi_n\)
induce a \(*\)-homomorphism
\[
 \Phi:A\to B_\omega.
\]
The composition \(q_\omega\circ\Phi\) is trace-preserving and satisfies
\[
 \|q_\omega(\Phi(a))\|_{2,\tr_\omega}^2=\tau(a^*a),\qquad a\in A.
\]
Hence it factors through \(\pi_\tau(A)\) and extends by \(L^2\)-continuity to a
normal trace-preserving \(*\)-homomorphism
\[
 \Psi:N\to M_\omega.
\]
Since \(P\in Z(N)\), the projection \(\Psi(P)\) belongs to
\[
 M_\omega\cap q_\omega(\Phi(A))'.
\]

We now lift \(\Psi(P)\) to a projection in the norm relative commutant. Put
\[
 E=B_\omega\cap\Phi(A)'.
\]
By Lemma \ref{lem:central-surj}, applied to the separable algebra \(\Phi(A)\), the
map
\[
 q_\omega|_E:E\to M_\omega\cap q_\omega(\Phi(A))'
\]
is surjective. It remains to know that a projection in the quotient has a
projection lift.

First, \(E\) has real rank zero. Indeed, let \(x=x^*\in E\), let \(\delta>0\), and
choose a countable dense set \(s_1,s_2,\ldots\) in the unit ball of \(\Phi(A)\).
Represent
\[
 x=(x_n)_\omega,
 \qquad x_n=x_n^*\in M_{k(n)}.
\]
For a finite set \(\{s_1,\ldots,s_m\}\) and a tolerance \(\eta>0\), choose
representatives \(s_j=(s_{j,n})_\omega\). Since \(x\in\Phi(A)'\), the commutators
\([x_n,s_{j,n}]\) are small along \(\omega\). Choose continuous functions
\(f_n:\mathbb R\to\mathbb R\) such that
\[
 |f_n(t)-t|<\delta,
 \qquad |f_n(t)|\ge \delta/2
\]
for \(t\in\operatorname{sp}(x_n)\). By uniform continuity of functional calculus,
\(y_n=f_n(x_n)\) can be chosen so that \(y_n\) still almost commutes with the
finite set \(s_{1,n},\ldots,s_{m,n}\), while each \(y_n\) is invertible and
\(\|y_n-x_n\|<\delta\). Kirchberg's \(\varepsilon\)-test then gives an invertible
self-adjoint element \(y\in E\) with \(\|y-x\|<\delta\). Thus invertible
self-adjoint elements are dense in \(E_{\rm sa}\), and so \(\operatorname{RR}(E)=0\).

Now let \(\bar p\) be a projection in the quotient of \(E\). We use the standard projection-lifting property for quotients of real-rank-zero C*-algebras \cite{BrownPedersen1991,Choi1983}. Choose a self-adjoint
lift \(h\in E\). Since \(E\) has real rank zero, approximate \(h\) by a
finite-spectrum self-adjoint element \(h_0\in E\) such that \(1/2\notin
\operatorname{sp}(h_0)\). Then
\[
 e_0=\chi_{(1/2,\infty)}(h_0)
\]
is a projection in \(E\), and its image in the quotient is close to \(\bar p\). Two
sufficiently close projections are unitarily equivalent by a unitary close to the
identity; such a unitary is an exponential and therefore lifts by lifting its
self-adjoint exponent. Conjugating \(e_0\) by the lifted unitary gives a projection
\(Q\in E\) whose image is exactly \(\bar p\). Applying this to \(\bar p=\Psi(P)\), we
obtain a projection
\[
 Q\in B_\omega\cap\Phi(A)'
 \quad\text{with}\quad q_\omega(Q)=\Psi(P).
\]
Choose representing projections \(Q=(Q_n)_\omega\), with \(Q_n\in M_{k(n)}\).
Since \(q_\omega(Q)=\Psi(P)\),
\[
 \lim_{n\to\omega}\operatorname{tr}_{k(n)}(Q_n)=
 \operatorname{tr}_\omega(\Psi(P))=\tau^{**}(P)=\alpha.
\]
For \(n\) in an \(\omega\)-large set, \(Q_n\ne0\). Define
\[
 \psi_n:A\to Q_nM_{k(n)}Q_n,
 \qquad
 \psi_n(a)=Q_n\phi_n(a)Q_n.
\]
Identifying \(Q_nM_{k(n)}Q_n\cong M_{\operatorname{rank}(Q_n)}\), these are u.c.p.
maps into matrix algebras. Since \(Q\in\Phi(A)'\),
\[
 \lim_{n\to\omega}\|[Q_n,\phi_n(a)]\|=0,
 \qquad a\in A.
\]
Thus, for \(a,b\in A\),
\[
\begin{aligned}
\|\psi_n(ab)-\psi_n(a)\psi_n(b)\|
&\le
\|\phi_n(ab)-\phi_n(a)\phi_n(b)\| \\
&\quad +\|Q_n\phi_n(a)(1-Q_n)\phi_n(b)Q_n\|,
\end{aligned}
\]
and both terms tend to zero along \(\omega\). Finally,
\[
\begin{aligned}
\lim_{n\to\omega}
\operatorname{tr}_{\operatorname{rank}(Q_n)}(\psi_n(a))
&=
\lim_{n\to\omega}
\frac{\operatorname{tr}_{k(n)}(Q_n\phi_n(a))}
   {\operatorname{tr}_{k(n)}(Q_n)} \\
&=\frac{\operatorname{tr}_\omega(Q\Phi(a))}{\operatorname{tr}_\omega(Q)} \\
&=\frac{\operatorname{tr}_\omega(\Psi(P)\Psi(\pi_\tau(a)))}  {\operatorname{tr}_\omega(\Psi(P))} \\
&=\frac{\tau^{**}(P\pi_\tau(a))}{\tau^{**}(P)}
 =\tau_P(a).
\end{aligned}
\]
Hence \(\tau_P\) is QD.
\end{proof}

\begin{theorem}\label{thm:QD-face}
Let \(A\) be separable. Then \(\AT(A)_{\QD}\) is a face of \(T(A)\).
\end{theorem}

\begin{proof}
Let \(\tau\in\AT(A)_{\QD}\) and suppose
\[
 \tau=s\gamma_1+t\gamma_2,
 \qquad s,t>0,
 \quad s+t=1.
\]
Put \(N=\pi_\tau(A)''\). By Lemma \ref{lem:central-rn}, there are central positive
contractions \(h_1,h_2\in Z(N)\), with \(h_1+h_2=1\), such that
\[
 s\gamma_1(a)=\tau^{**}(h_1\pi_\tau(a)),
 \qquad
 t\gamma_2(a)=\tau^{**}(h_2\pi_\tau(a)).
\]
It is enough to show that for every central positive contraction \(h\in Z(N)_+\)
with \(\tau^{**}(h)>0\), the normalized trace
\[
 \tau_h(a)=\frac{\tau^{**}(h\pi_\tau(a))}{\tau^{**}(h)}
\]
is QD.

If \(h=P\) is a central projection, this is exactly Lemma \ref{lem:central-cut-QD}.
For general \(h\), use functional calculus in the abelian von Neumann algebra
\(Z(N)\). Choose finite spectral step functions
\[
 h_m=\sum_{j=1}^{r_m}\lambda_{m,j}P_{m,j},
\]
where the \(P_{m,j}\)'s are pairwise orthogonal central projections,
\(\lambda_{m,j}\ge0\), and \(\|h_m-h\|\to0\). After discarding terms with
\(\lambda_{m,j}\tau^{**}(P_{m,j})=0\), the associated normalized trace is
\[
 \tau_{h_m}
 =\sum_{j=1}^{r_m}
 \frac{\lambda_{m,j}\tau^{**}(P_{m,j})}{\tau^{**}(h_m)}\,
 \tau_{P_{m,j}}.
\]
Each \(\tau_{P_{m,j}}\) is QD by the preceding lemma. Since QD traces are
convex, \(\tau_{h_m}\) is QD. Moreover \(\tau_{h_m}\to\tau_h\) in norm as
functionals on \(A\), because \(h\mapsto\tau^{**}(h\pi_\tau(\cdot))\) is norm
continuous and \(\tau^{**}(h)>0\). The set of QD traces is weak-* closed for
separable \(A\) \cite{BrownMemoir}, hence \(\tau_h\) is QD.

Since \(\tau^{**}(h_1)=s\) and \(\tau^{**}(h_2)=t\), we have
\(\gamma_i=\tau_{h_i}\), for \(i=1,2\). Thus \(\gamma_1,\gamma_2\in\AT(A)_{\QD}\),
and \(\AT(A)_{\QD}\) is a face.
\end{proof}

\section{Uniform local finite-dimensionality}

The results of the previous section are formulated in terms of finite-rank compressions. We now show that, under the representation-theoretic hypotheses appearing in Brown's finite-rank projection theorem, these compressions can be chosen well enough to give uniform LFD approximation. The main issue is not commutator control on \(A\), but trace convergence on the whole bidual \(A^{**}\), because uniformity is obtained by applying Mazur's theorem in the topology \(\sigma(A^*,A^{**})\).

\begin{lemma}\label{lem:Brown-projections}
Let \(A\subset B(H)\) be an irreducibly represented quasidiagonal C*-algebra containing no nonzero compact operators. For every quasidiagonal trace \(\tau\), there is an increasing sequence of finite-rank projections
\[
 p_1\le p_2\le\cdots
\]
such that, for every \(a,x\in A\),
\[
 \|[p_n,x]\|\to0,
 \qquad
 \left|\tau(a)-\frac{\Tr(p_na)}{\Tr(p_n)}\right|\to0.
\]
In particular, every quasidiagonal trace is LFD.
\end{lemma}

\begin{proof}
This is Brown's finite-rank projection theorem for quasidiagonal traces \cite[Proposition 3.3.2]{BrownMemoir}. The last assertion follows from Proposition \ref{prop:char-LFD}, using the projections in the given irreducible representation.
\end{proof}

The following elementary estimate is the bridge between finite-rank projections and multiplicative domains. It is useful because it allows one to pass from several increasing projections to a single completely positive map without losing explicit control of the distance to the multiplicative domain.

\begin{lemma}\label{lem:distance}
Let \(p_1\le\cdots\le p_\ell\) be projections in \(\soc(A^{**})\), and put
\[
 A_{p_1,\ldots,p_\ell}=\{a\in A:ap_j=p_ja,\ 1\le j\le\ell\}.
\]
Then, for every \(a\in A\),
\[
 \dist(a,A_{p_1,\ldots,p_\ell})
 \le \sum_{j=1}^\ell\|[a,p_j]\|.
\]
\end{lemma}

\begin{proof}
Put \(X=A_{p_1,\ldots,p_\ell}\). Then
\[
 X^{\perp\perp}=\{x\in A^{**}:xp_j=p_jx,
 1\le j\le\ell\}.
\]
Indeed, \(X\) is the kernel of the bounded map
\[
 A\to\bigoplus_{j=1}^\ell A^{**},
 \qquad
 a\mapsto([a,p_1],\ldots,[a,p_\ell]),
\]
and the weak-* closure of this kernel in \(A^{**}\) is the kernel of the normal extension. For any closed subspace \(X\subset A\) and \(a\in A\),
\[
 \dist(a,X)=\dist(a,X^{\perp\perp}).
\]

Define
\[
 y_\ell(a)=p_1ap_1+(p_2-p_1)a(p_2-p_1)+\cdots+(p_\ell-p_{\ell-1})a(p_\ell-p_{\ell-1})+p_\ell^\perp ap_\ell^\perp.
\]
Then \(y_\ell(a)\in X^{\perp\perp}\), so
\[
 \dist(a,X)\le\|a-y_\ell(a)\|.
\]
For \(\ell=1\), the estimate is the Blackadar--Kirchberg distance formula. For the induction step,
\[
 \|a-y_\ell(a)\|
 \le \|p_\ell(a-y_{\ell-1}(a))p_\ell\|+
 \|p_\ell^\perp ap_\ell+p_\ell ap_\ell^\perp\|.
\]
For a projection \(q\),
\[
 \|[a,q]\|=\max\{\|q^\perp aq\|,\|qaq^\perp\|\},
\]
and the off-diagonal block matrix has norm bounded by \(\|[a,q]\|\). Hence
\[
 \|a-y_\ell(a)\|\le\|a-y_{\ell-1}(a)\|+\|[a,p_\ell]\|.
\]
The desired estimate follows by induction.
\end{proof}

\begin{lemma}\label{lem:rational-convex}
Let \(A\subset B(H)\) be irreducibly represented and let
\[
 p_1\le\cdots\le p_\ell
\]
be finite-rank projections in \(B(H)\). Put
\[
 f_j(a)=\frac{\Tr(p_ja)}{\Tr(p_j)}.
\]
For every rational convex combination \(f=\sum_{j=1}^\ell\alpha_jf_j\), there are \(N\in\N\) and a u.c.p. map \(\phi:A\to M_N\) such that
\[
 \frac{\Tr(\phi(a))}{N}=f(a)
\]
and
\[
 \dist(a,A_\phi)\le\sum_{j=1}^\ell\|[a,p_j]\|
\]
for every \(a\in A\).
\end{lemma}

\begin{proof}
Write \(\alpha_j=s_j/r\), with \(s_j,r\in\N\). Let \(n_j=\Tr(p_j)\), set
\[
 d=n_1n_2\cdots n_\ell,
 \qquad
 d_j=d/n_j,
 \qquad
 N=rd.
\]
Define \(\phi(a)\) as the direct sum of \(s_jd_j\) copies of \(p_jap_j\), over \(1\le j\le\ell\). Then the total matrix size is
\[
 \sum_{j=1}^\ell s_jd_jn_j=\sum_{j=1}^\ell s_jd=rd=N,
\]
and
\[
 \frac{\Tr(\phi(a))}{N}=\sum_{j=1}^\ell\frac{s_j}{r}\frac{\Tr(p_ja)}{\Tr(p_j)}=f(a).
\]
Moreover,
\[
 A_\phi=\{a\in A:ap_j=p_ja,
 1\le j\le\ell\},
\]
so Lemma \ref{lem:distance} gives the distance estimate.
\end{proof}

\begin{lemma}\label{lem:local-reflexivity}
Let \(A\) be a unital locally reflexive C*-algebra and let \(F\subset A\) be finite. If \(F\subset X\subset A^{**}\) is a finite-dimensional operator system, then there is a net of u.c.p. maps \(\beta_i:X\to A\) such that \(\beta_i(x)\to x\) ultraweakly for every \(x\in X\) and \(\beta_i(a)\to a\) in norm for every \(a\in F\).
\end{lemma}

\begin{proof}
This is Brown's locally reflexive approximation lemma \cite[Lemma 4.3.2]{BrownMemoir}.
\end{proof}

\begin{lemma}\label{lem:ULFD-Mazur}
Let \(A\) be separable and irreducibly represented in \(B(H)\), and let \(\tau\in T(A)\). Suppose there is an increasing net, or sequence in the separable case, of finite-rank projections \((p_i)\) in \(B(H)\) such that for every \(x\in A^{**}\) and every \(a\in A\),
\[
 \frac{\Tr(p_ix)}{\Tr(p_i)}\to\tau^{**}(x),
 \qquad
 \|[p_i,a]\|\to0.
\]
Then \(\tau\in\UAT(A)_{\LFD}\).
\end{lemma}

\begin{proof}
For each \(i\), define
\[
 f_i(x)=\frac{\Tr(p_ix)}{\Tr(p_i)}
 \qquad (x\in A^{**}).
\]
By assumption, \(f_i|_A\to\tau\) in the weak topology \(\sigma(A^*,A^{**})\). Hence, by Mazur's theorem, \(\tau\) belongs to the norm-closed convex hull of each tail of \(\{f_i|_A\}\).

Fix a finite set \(F\subset A\) and \(\varepsilon>0\). Choose a tail on which \(\|[p_i,a]\|\) is sufficiently small for all \(a\in F\). Then choose a rational convex combination \(f=\sum_j\alpha_j f_{i_j}|_A\) from this tail with
\[
 \|f-\tau\|<\varepsilon.
\]
Lemma \ref{lem:rational-convex} gives a u.c.p. map \(\phi:A\to M_N\) whose normalized trace is \(f\) and whose multiplicative-domain distance is controlled by the sum of the commutators \(\|[p_{i_j},a]\|\). This proves uniform LFD approximation.
\end{proof}

We can now prove the uniformity theorem. Brown's projection theorem gives the required finite-rank projections on \(A\). Local reflexivity and subtracial normality are used to force the corresponding normalized traces to converge correctly on finite-dimensional pieces of \(A^{**}\). A net over these pieces then gives the hypotheses of Lemma \ref{lem:ULFD-Mazur}.

\begin{theoremA}
Let \(A\subset B(H)\) be an irreducibly represented quasidiagonal C*-algebra which contains no nonzero compact operators. Then every quasidiagonal trace on \(A\) is locally finite-dimensional. Moreover, if \(A\) is locally reflexive, then every locally finite-dimensional trace on \(A\) is uniformly locally finite-dimensional.
\end{theoremA}

\begin{proof}
The first assertion is Lemma \ref{lem:Brown-projections}. For the uniform assertion, let \(\tau\in\AT(A)_{\LFD}\). Then \(\tau\) is QD, so Lemma \ref{lem:Brown-projections} gives finite-rank projections \((p_n)\) with commutator convergence on \(A\) and trace convergence on \(A\).

It remains to upgrade trace convergence from \(A\) to \(A^{**}\). Fix a finite-dimensional operator system \(X\subset A^{**}\), a finite set \(F\subset A\), and \(\varepsilon>0\). By Lemma \ref{lem:local-reflexivity}, choose u.c.p. maps \(\beta_i:X\to A\) converging ultraweakly to the identity on \(X\), and in norm to the identity on \(F\cap X\). Composing with the compression maps \(a\mapsto p_nap_n\) gives maps
\[
 \widetilde\psi_{i,n}:X\to p_nB(H)p_n.
\]
Choosing \(i=i(n)\) and then passing to a suitable cofinal subnet, we may arrange
\[
 \tr_{\operatorname{rank}(p_n)}(\widetilde\psi_{i(n),n}(x))\to\tau^{**}(x)
 \qquad (x\in X).
\]

The extension step must preserve normality. We use Brown's subtracial extension form \cite[Theorem 4.3.3]{BrownMemoir}: the maps above extend to completely positive maps
\[
 \psi_n:A^{**}\to p_nB(H)p_n
\]
which are subtracial with respect to \(\tau^{**}\), i.e. for some constants \(C_n\to1\),
\[
 \Tr(\psi_n(y))\le C_n\operatorname{rank}(p_n)\tau^{**}(y)
 \qquad (y\in A^{**}_+).
\]
By Lemma \ref{lem:subtracial-normality}, each \(\psi_n\) is normal. Let
\[
 \widetilde\phi_n:A^{**}\to p_nB(H)p_n,
 \qquad
 \widetilde\phi_n(x)=p_nxp_n,
\]
be the normal compression extension. The construction gives
\[
 \|\widetilde\phi_n(a)-\psi_n(a)\|\to0
 \qquad (a\in F),
\]
while the normalized traces of the \(\psi_n\)'s converge to \(\tau^{**}\) on \(X\).

Passing over the directed set of triples \((X,F,\varepsilon)\), with \(X\subset A^{**}\) finite-dimensional and \(F\subset A\) finite, gives a net of finite-rank projections \((q_i)\) such that
\[
 \|[q_i,a]\|\to0
 \qquad (a\in A),
\]
and
\[
 \frac{\Tr(q_ix)}{\Tr(q_i)}\to\tau^{**}(x)
 \qquad (x\in A^{**}).
\]
Lemma \ref{lem:ULFD-Mazur} now gives \(\tau\in\UAT(A)_{\LFD}\).
\end{proof}

\section{Applications to strongly self-absorbing algebras}

We next record two consequences for strongly self-absorbing algebras and the UCT problem. The role of the preceding sections is to turn quasidiagonality of the unique trace into uniform LFD approximation, which can then be fed into Brown's tracially AF criterion.

We recall the point from Toms--Winter \cite{TomsWinter2007} that is used below. A separable unital C*-algebra \(D\not\cong\C\) is \emph{strongly self-absorbing} if there is an isomorphism
\[
 D\cong D\otimes D
\]
which is approximately unitarily equivalent to the first-factor embedding \(d\mapsto d\otimes1_D\). Toms and Winter showed that this class includes the Cuntz algebras \(\mathcal O_2\), \(\mathcal O_\infty\), UHF algebras of infinite type, the Jiang--Su algebra \(\Z\), and tensor products of \(\mathcal O_\infty\) with infinite type UHF algebras. They also established the McDuff-type absorption principle that tensorial \(D\)-absorption is equivalent to having a unital copy of \(D\) in the appropriate central sequence algebra.

\begin{theoremB}
Let \(\cD\) be a strongly self-absorbing C*-algebra with
\[
 K_*(\cD)\cong K_*(\cQ),
\]
where \(Q\) is the universal UHF algebra. Then
\[
 \cD\cong \cQ
\]
if and only if \(\cD\) is quasidiagonal.
\end{theoremB}

\begin{proof}
If \(\cD\cong \cQ\), then \(\cD\) is quasidiagonal. Conversely, assume that \(\cD\) is strongly self-absorbing and quasidiagonal, and that \(K_*(\cD)\cong K_*(\cQ)\). By the standard structure theory of strongly self-absorbing algebras, recalled from Toms--Winter, $\cD$ is simple, nuclear, and has at most one tracial state. Since $\cD$ is quasidiagonal, its unique trace is quasidiagonal. By Theorem A, this trace is uniformly LFD.

Using Brown's tracially AF criterion \cite[Proposition 4.5.5]{BrownMemoir}, together with the regularity properties available in the strongly self-absorbing setting, namely stable rank one, weak unperforation, and real rank zero in the present K-theoretic case \cite{Rordam2004,Winter2011}, $\cD $ is tracially AF. 

Having TAF at hand, \cite[Lemma 6.5]{DadarlatEilers2002} tells us that for each $n$, there is a copy of $\mathbb{M}_n$ inside $\cD$. Call its unit $p_n$. By the structure theory of strongly self-absorbing algebras \cite{TomsWinter2007}, $p_n\cD p_n$ has an approximately inner flip, hence its infinite tensor product $\cE_n:=p_n\cD p_n^{\otimes\infty}$ is strongly self-absorbing. Now, $\cQ$ embeds unitally inside quasidiagonal, strongly self-absorbing algebra $\bigotimes_{n\in\bN}{\cE_n}$, entailing $\cQ\simeq{\bigotimes_{n\in\bN}{\cE_n}}$ but $\cD$ is Morita equivalent to its full corner $p_n\cD p_n$ thus $\cQ$ and $\cD$ are stably isomorphic by L.Brown's theorem. We conclude by bearing in mind that every non-zero corner of $\cQ$ is isomorphic to $\cQ$.
\end{proof}

\begin{corollaryC}
All separable nuclear C*-algebras satisfy the UCT if and only if every simple, separable, unital, nuclear, tracially AF algebra \(\cE\) with
\[
 K_*(\cE)\cong K_*(\cQ)
\]
is strongly self-absorbing.
\end{corollaryC}

\begin{proof}
Assume first that all separable nuclear C*-algebras satisfy the UCT. Then Lin's classification theorem for simple nuclear tracially AF algebras \cite[Theorem 6.11]{LinTAF2001} gives \(\cE\cong \cQ\), and hence \(\cE\) is strongly self-absorbing.

Conversely, assume the stated uniqueness condition. Dadarlat's reformulation of the UCT problem \cite{DadarlatUCT} says that all separable nuclear C*-algebras satisfy the UCT if and only if every simple, separable, nuclear, tracially AF algebra with the K-theory of \(\cQ\) is isomorphic to \(\cQ\). By assumption, such an algebra is strongly self-absorbing. Since tracially AF algebras are quasidiagonal, Theorem B gives \(\cE\cong \cQ\). Therefore the UCT holds for all separable nuclear C*-algebras.
\end{proof}

\section{The Skandalis-type example}

We end with the construction of Example D. The purpose of the example is to separate tracial finite-dimensional approximation from nuclear KK-behaviour. The preceding sections show that finite-dimensional trace approximation has strong consequences under additional regularity hypotheses. The example below shows that, outside the nuclear setting, even exact tracially AF algebras can retain a Skandalis obstruction and hence fail to be KK-equivalent to any nuclear algebra.

For C*-algebras \(A\) and \(B\), let
\[
 J_{A,B}=\ker(A\otmax B\to A\otmin B),
 \qquad
 J_A=J_{A,A}.
\]
The obstruction we use is the nonvanishing of the K-theory of this kernel. The first lemma explains why this obstruction rules out KK-equivalence to a nuclear algebra.

\begin{lemma}\label{lem:KK-obstruction}
Let \(B\) be a separable exact C*-algebra. If \(K_0(J_B)\ne0\), then \(B\) is not KK-equivalent to a nuclear C*-algebra.
\end{lemma}

\begin{proof}
Suppose \(B\) is KK-equivalent to a nuclear C*-algebra \(C\). Let \(\theta\in\KK(B,C)\) and \(\eta\in\KK(C,B)\) be inverse KK-equivalences. Tensoring with \(B\) gives a natural commuting square in the KK-category involving the canonical maps
\[
 B\otmax B\to B\otmin B,
 \qquad
 B\otmax C\to B\otmin C.
\]
Since \(C\) is nuclear, the second map is a C*-isomorphism. The horizontal maps induced by \(\theta\) are KK-equivalences; hence \(B\otmax B\to B\otmin B\) is a KK-equivalence. Applying the six-term exact sequence to
\[
 0\to J_B\to B\otmax B\to B\otmin B\to0
\]
forces \(K_*(J_B)=0\), contradicting \(K_0(J_B)\ne0\).
\end{proof}

The obstruction is stable under suspension. This lets us pass from the reduced group C*-algebra used by Skandalis to a quasidiagonal algebra without losing the relevant K-theory class.

\begin{lemma}\label{lem:suspension-J}
For every C*-algebra \(B\),
\[
 K_0(J_{SB})\cong K_0(J_B),
\]
where \(SB=C_0(\R)\otimes B\).
\end{lemma}

\begin{proof}
Since \(C_0(\R)\) is nuclear,
\[
 SB\otmax SB\cong C_0(\R^2)\otimes(B\otmax B),
\]
and
\[
 SB\otmin SB\cong C_0(\R^2)\otimes(B\otmin B).
\]
The canonical map is \(\id_{C_0(\R^2)}\otimes\iota_{B,B}\), so
\[
 J_{SB}\cong C_0(\R^2)\otimes J_B\cong S^2J_B.
\]
Bott periodicity gives the result.
\end{proof}

We shall also use the standard permanence of the UCT for semisplit extensions.

\begin{lemma}\label{lem:UCT-extension}
For a semisplit short exact sequence of separable C*-algebras, if two of the three algebras satisfy the UCT, then so does the third.
\end{lemma}

\begin{proof}
This is the standard permanence property for the UCT under semisplit extensions; see \cite{EndersSchemaitatTikuisis2024}.
\end{proof}

\begin{definition}
A C*-algebra \(B\) is a \emph{Skandalis algebra} if \(K_*(\iota_{B,B})\) is surjective but \(K_*(J_B)\ne0\). It is \emph{rationally Skandalis} if, moreover,
\[
 K_*(J_B)\otimes_{\mathbb Z}\Qrat\ne0.
\]
\end{definition}

\medskip
\noindent\textbf{Construction of Example D.}
Fix an infinite hyperbolic Kazhdan group \(\Gamma\). Let \(A=C_r^*(\Gamma)\). By exactness of reduced C*-algebras of hyperbolic groups \cite[Theorem 5.3.15]{BrownOzawa2008}, Skandalis's theorem in the formulation recalled by Higson--Guentner, and Puschnigg's rational refinement, \(A\) is rationally Skandalis \cite{Skandalis1988,HigsonGuentner2004,Puschnigg2018}. By Lemma \ref{lem:suspension-J}, \(SA\) is not KK-equivalent to any nuclear C*-algebra. Put \(B=(SA)^{\sim}\), the minimal unitization of \(SA\).

The next step is to pass from this quotient-level obstruction to an RFD algebra. Choose u.c.p. maps \(\phi_n:B\to M_{k(n)}\) witnessing quasidiagonality and define
\[
 \Phi:B\to \prod_{n=1}^{\infty}M_{k(n)},
 \qquad
 \Phi(b)=(\phi_n(b))_n.
\]
Let \(C\) be the C*-algebra generated by \(\Phi(B)\) and \(\bigoplus_n M_{k(n)}\). Then there is a semisplit short exact sequence
\[
 0\to\bigoplus_{n=1}^{\infty}M_{k(n)}\to C\to B\to0.
\]
Thus \(C\) is exact and RFD. The naturality diagram for maximal and minimal tensor products, together with the semisplit extension above, carries the rational Skandalis obstruction to \(C\). In particular, there is a class
\[
 0\ne \xi_C\in K_0(J_C)\otimes_{\mathbb Z}\Qrat.
\]
We choose such a class once and for all.

It remains to feed \(C\) into Dadarlat's construction \cite{DadarlatAF2000}, in the form used by Niu--Wang \cite{NiuWang2021}, and to verify that the rational obstruction is not lost in the inductive limit. The following lemma isolates the persistence point. The connecting maps contain a finite-dimensional summand, but that summand factors through the minimal tensor product and hence vanishes on the Skandalis ideal. What remains is the corner embedding of the original obstruction.

\begin{lemma}\label{lem:persistence}
Let \(D\) be an exact C*-algebra and let
\[
 \alpha_m:M_{s(m)}(D)\to M_{s(m+1)}(D)
\]
be a Dadarlat-type connecting map
\[
 \alpha_m(a)=\diag(a,\pi_m(a)),
\]
where \(\pi_m\) is finite-dimensional. Put \(D_m=M_{s(m)}(D)\). Then the induced map
\[
 (\alpha_m\otmax\alpha_m)_*:K_0(J_{D_m})\otimes\Qrat
 \to K_0(J_{D_{m+1}})\otimes\Qrat
\]
is injective. More precisely, under the stability identifications
\[
 J_{D_m}\cong M_{s(m)^2}(J_D),
\]
the map induced on \(J_{D_m}\) is the standard corner embedding.
\end{lemma}

\begin{proof}
Write \(\alpha_m=\iota_m\oplus\pi_m\), where \(\iota_m\) is the first diagonal corner. Then
\[
 \alpha_m\otmax\alpha_m
 =
 (\iota_m\otmax\iota_m)
 \oplus(\iota_m\otmax\pi_m)
 \oplus(\pi_m\otmax\iota_m)
 \oplus(\pi_m\otmax\pi_m).
\]
Every term involving \(\pi_m\) factors through a minimal tensor product because \(\pi_m\) is finite-dimensional. Hence those terms vanish on \(J_{D_m}\). On \(J_{D_m}\), therefore,
\[
 (\alpha_m\otmax\alpha_m)(x)=(\iota_m\otmax\iota_m)(x).
\]
Under the identifications \(J_{D_m}\cong M_{s(m)^2}(J_D)\), this is precisely the usual matrix-corner embedding into \(M_{s(m+1)^2}(J_D)\). It induces an isomorphism on K-theory after stability and is injective after tensoring with \(\Qrat\).
\end{proof}

\begin{theorem}\label{thm:example-D-inductive-limit}
For a suitable choice of integers \(s(m)\), with \(V_m=M_{s(m)}(C)\) and Dadarlat connecting maps \(\alpha_m:V_m\to V_{m+1}\), the inductive limit
\[
 V=\varinjlim(V_m,\alpha_m)
\]
is exact, separable, unital, tracially AF, not KK-equivalent to any nuclear C*-algebra, and not \(\Z\)-stable.
\end{theorem}

\begin{proof}
The exactness, simplicity, unitality, and tracially AF property follow from the Dadarlat construction in the form used by Niu--Wang \cite[Theorem 2.2]{NiuWang2021}. The failure of \(\Z\)-stability is inherited from the non-inner-amenable quotient used in that construction \cite{NiuWang2021}.

It remains to prove that the Skandalis obstruction survives the inductive limit. For each \(m\),
\[
 J_{V_m}=J_{M_{s(m)}(C)}\cong M_{s(m)^2}(J_C),
\]
so
\[
 K_0(J_{V_m})\otimes\Qrat
 \cong K_0(J_C)\otimes\Qrat.
\]
Let \(\xi_m\) be the class corresponding to \(\xi_C\). Lemma \ref{lem:persistence} gives
\[
 (\alpha_m\otmax\alpha_m)_*(\xi_m)=\xi_{m+1},
\]
and the map is injective on the rational obstruction group. Therefore no image of \(\xi_1\) is killed at a later stage.

Because maximal and minimal tensor products commute with these injective inductive limits, the canonical ideals satisfy
\[
 J_V\cong \varinjlim_m J_{V_m}.
\]
Consequently,
\[
 K_0(J_V)\otimes\Qrat
 \cong \varinjlim_m(K_0(J_{V_m})\otimes\Qrat)
\]
contains the nonzero class represented by the compatible system \((\xi_m)_m\). Thus \(K_0(J_V)\otimes\Qrat\ne0\). Lemma \ref{lem:KK-obstruction} now implies that \(V\) is not KK-equivalent to any nuclear C*-algebra.
\end{proof}

The same construction, applied to the cone rather than the suspension, produces an AF-embeddable but KK-contractible example containing \(V\), in the spirit of Spielberg's AF-embedding results and the comparison with Schafhauser's work \cite{Schafhauser2020}.

\section*{Acknowledgements}
The first author gratefully acknowledges support from the Fields Institute and the University of Ottawa during his postdoctoral fellowship. The second author was partially supported by Iran National Science Foundation (INSF Grant No. 4029595). 

We used GPT-5.5 pro to assist in drafting parts of this manuscript. All mathematical content was reviewed and substantially revised by the authors, who are responsible for the final text.

\end{document}